\documentclass[preprint,sort&compress]{myarticle}

\usepackage{amssymb,graphicx,amsthm,amsbsy,float}
\pdfoutput=1 

\newcommand{\ZZ}{\mathbb{Z}}
\newcommand{\RR}{\mathbb{R}}
\newcommand{\NN}{\mathbb{N}}
\newcommand{\CC}{\mathbb{C}}

\newcommand{\Ab}{{\boldsymbol{A}}}
\newcommand{\Bb}{{\boldsymbol{B}}}
\newcommand{\Cb}{{\boldsymbol{C}}}

\newcommand{\Eb}{{\boldsymbol{E}}}
\newcommand{\Fb}{{\boldsymbol{F}}}
\newcommand{\Gb}{{\boldsymbol{G}}}
\newcommand{\Hb}{{\boldsymbol{H}}}
\newcommand{\Ib}{{\boldsymbol{I}}}

\newcommand{\Lb}{{\boldsymbol{L}}}
\newcommand{\Mb}{{\boldsymbol{M}}}

\newcommand{\Pb}{{\boldsymbol{P}}}

\newcommand{\Ub}{{\boldsymbol{U}}}
\newcommand{\Vb}{{\boldsymbol{V}}}

\newcommand{\Xb}{{\boldsymbol{X}}}

\newcommand{\ab}{{\boldsymbol{a}}}
\newcommand{\bb}{{\boldsymbol{b}}}

\newcommand{\eb}{{\boldsymbol{e}}}

\newcommand{\pb}{{\boldsymbol{p}}}

\newcommand{\ssb}{{\boldsymbol{s}}}

\newcommand{\gammab}{{\boldsymbol{\gamma}}}
\newcommand{\xib}{{\boldsymbol{\xi}}}

\newcommand{\BZero}{{\boldsymbol{0}}}
\newtheorem{theorem}{Theorem}
\begin{document}

\begin{frontmatter}

\title{Partial parameterization of orthogonal\\ wavelet matrix filters\let\thefootnote\relax\footnotetext{\emph{''NOTICE: this is the authors' version of a work that was accepted for publication 
			in Journal of Applied and Computational Mathematics (doi:10.1016/j.cam.2012.11.016). 
			Changes resulting from the publishing process, such as peer review, editing, corrections,
			structural formatting, and other quality control mechanisms may not be reflected in this document. Changes may have been made to this work since it was submitted for publication''}}}

\author{Mariantonia Cotronei%\corref{cor1}
}
%\ead{mariantonia.cotronei@unirc.it}
\address{DIMET, 
Universit\`a Mediterranea di Reggio Calabria,\\Via Graziella loc. Feo di Vito, I-89122 Reggio Calabria, Italy\\
mariantonia.cotronei@unirc.it}

\author{Matthias Holschneider}
%\ead{hols@math.uni-potsdam.de}

\address{Institut f\"ur Mathematik, Universit\"at Potsdam,\\ Am Neuen Palais 10,
D-14469, Potsdam,
Germany\\
hols@math.uni-potsdam.de}


\begin{abstract}
In this paper we propose a procedure which allows the construction of a large family of FIR 
$d\times d$ matrix wavelet filters  by exploiting the  one-to-one correspondence between QMF systems and orthogonal operators which commute
with the shifts by two.  A characterization
of the class of filters of full rank type that can be obtained with such procedure is given.
In particular, we  restrict our attention to a special construction  based on the representation of $SO(2d)$ in terms of the elements of its Lie algebra. Explicit expressions for the filters in the case $d = 2$ are given,
as a result of a local analysis of the parameterization obtained from perturbing the Haar
system.
\end{abstract}

\begin{keyword}
Full rank matrix filters \sep Multichannel wavelets \sep Quadrature Mirror Filters \sep Vector subdivision schemes
\MSC 65T60

\end{keyword}

\end{frontmatter}


\section{Introduction}
Parameterization of orthogonal and biorthogonal filters has been an important topic of research in the context of wavelet analysis (for the relation between
filters and wavelets see e.g. \cite{Wutam}). Pioneering works in this area have been carried out, for example, by Pollen \cite{Pollen} and  Holschneider \cite{holschneider}, who proposed parameterizations based on \emph{loop group factorization}. 
Later Sweldens \cite{Sweldens}  introduced the \emph{lifting scheme} which is now widely used to construct 
biorthogonal  families of filters. 
For yet another interesting  approach to the explicit construction of  wavelet filters based on their correspondence with representations of the so-called Cuntz relations we refer the reader, for instance, to \cite{Jorgensen}. 
\par
The aim of this paper is to extend the approach given
in \cite{holschneider} to the construction of an infinite family of orthogonal
perfect reconstruction \emph{matrix filter banks} of \emph{full rank} type. These kinds of filters, introduced  in  \cite{BacchelliCotroneiSauer02a,ContiCotronei,ContiCotroneiSauer07,ContiCotroneiSauer07S,CotroneiSauer07}, are associated to the \emph{multichannel  wavelet analysis} of signals which are vector-valued, consisting of several components typically associated to different channels. 
Signals of this type arise naturally in many application contexts; typical examples are: brain activity (EEG/MEG) data, colour images, multisensor data,  financial time series, etc. 
A classical scalar wavelet analysis   applied the single components of this kind of data might not be appropriate, because it ignores the
possible correlation among  the channels. Multichannel wavelets provide a more effective tool, able to process such type of signals as ''complete'' vectors, so to extract peculiar information from the overall behaviour of the components and to reveal/exploit inter-correlations.
\par
The main challenge in the context of  multichannel wavelet analysis is the construction of matrix filters satisfying the \emph{quadrature mirror filter} (\emph{QMF}) condition. In particular, the direct construction of the filter associated to the matrix scaling function from the QMF  constraints  is not convenient because of their non-linear nature. On the other hand, a Daubechies-like approach only gives rise to filters  which are essentially diagonal, in the sense that they can be reduced, via similarity transformations, to  scalar schemes. In \cite{ContiCotroneiSauer07S} a construction of orthogonal matrix refinable functions has been proposed as a spectral factorization problem, based on the connection between  orthogonal and interpolatory matrix subdivision schemes. Nevertheless this approach cannot in principle give rise to any closed parameterized form of the filters.
	As to the construction of the corresponding wavelet filter, unlike the scalar case, there is no trivial alternating flip trick to express it  in terms of the scaling function filter, due to the non-commutative nature of the matrix filter context. The numerical scheme presented in \cite{ContiCotronei} is an example of effective numerical approach to the problem, which, nevertheless, relies on the a priori knowledge of the scaling filter.
\par
In this work, we address both the problems (construction of the scaling function and of the wavelet) at the same time, by deriving  a procedure which provides a natural and intrinsic characterization of (quasi) all matrix QMF  filters. The main idea is to exploit the
one-to-one correspondence between matrix QMF filters and orthogonal
operators that commute with translations by 2, so that starting with a trivial system and repeatedly   applying such type of operators it is possible to produce large classes of filters. 
\par 
It is important to point out that the QMF filter construction in the matrix case can be viewed also in a multiwavelet perspective (see, for example, \cite{GHM,MicchelliSauer,StrangStrela,XiaSuter}), but  all the constructions proposed for this kind of bases do not take into account the full rank requirement (since  multiwavelets still generate multiresolution analyses for spaces of  scalar rather than vector-valued functions). Our construction, on the other hand,  can be straightforwardly applied to the realization of multiwavelet bases.
\par
The paper is organized as follows. Section 2  presents some preliminaries and notation. 
In Section 3 we give a detailed exposition of our construction and a characterization result. 
Finally, in Section 4 we examine the particular case of dimension 2 providing 
 explicit descriptions of   families of  parameterized filters, in particular those associated to the representation  of the Lie  group  of rotations in terms of infinitesimal generators.

\section{Notation and basic facts}
%For $d\in \NN$ we  denote by $L^2(\RR)^d$ the space 
%of vector valued sequences taking values in $\RR^d$. Its scalar product is given by 
%($\ab=[\dots [a(k)^1,a(k)^2,\dots,a(k)^d]^t, [a(k+1)^1,a(k+1)^2,\dots,a(k+1)^d]^t\dots]$)
%$$
%<\ab,\bb>=\sum \sum_k \sum_{i=1}^d a(k)^i b(k)^i.
%$$

 By $L^2 \left( \ZZ
\right)^{d \times d}$, $d\in \NN$, we denote the space of square-summable $d \times d$--matrix valued sequences, that is
$$L^2 \left( \ZZ
\right)^{d \times d}=
\left\{
\Ab : \ZZ \to \RR^{d \times d}:\,
\left( \sum_{k \in \ZZ} \left| \Ab (k) 
  \right|_2^2 \right)^{1/2}<\infty\right\}
  $$
  where
$|\cdot|_2$ denotes the $2$-norm on the matrix space $\RR^{d\times d}$.  
For notational simplicity we write $L^2 \left( \ZZ
\right)^{d}$ for $L^2 \left( \ZZ
\right)^{d \times 1}$ and $L^2 (\ZZ)$ for $L^{2}
\left( \ZZ \right)^1$.
\par
For two  matrix sequences $\Ab,\Bb\in L^2 \left( \ZZ
\right)^{d \times d}$, we define the \emph{ inner product} as
$$\langle \Ab,\Bb \rangle =\sum_{k\in \ZZ} \Ab(k)^T\Bb(k)$$
and we say that the two sequences are
\emph{orthogonal} to each other if 
$$\langle \Ab,\Bb \rangle =\BZero,$$
where $\BZero$ denotes the $d\times d $ matrix all of whose entries are zero.
\par
In particular, given two vector sequences $\ab=\left([a_1(k),\ldots, a_d(k)]^T: \, k\in \ZZ\right) $, 
$\bb=\left([b_1(k),\ldots, b_d(k)]^T: \, k\in \ZZ\right) $ in $L^2 \left( \ZZ
\right)^{d}$, then their inner (scalar) product is:
$$<\ab,\bb>=\sum_{k\in \ZZ} \sum_{i=1}^d a_i(k) b_i(k).$$

\par
For $n\in \ZZ$, we define the \emph{Dirac delta sequence} $\delta_n$ as the sequence whose elements satisfy:
$$\delta_n(k)=\left\{ \begin{array}{ll}1, & k=n\\0,& k\ne n \end{array}\right.$$
\par
If $n\in \ZZ$,  we introduce  the \emph{translation} and  \emph{downsampling} operators $T_n$ and $D$ acting on a sequence
 $\Ab\in L^2(\ZZ)^{d\times d}$ respectively as:
$$ T_n \Ab= \Ab(\cdot -n),\quad  D\Ab=\Ab(2\cdot).$$
The transpose of the downsampling operator is the \emph{upsampling operator}, described by the action
$$(D^T\Ab)(k)=\left\{\begin{array}{ll}\Ab(k/2), & k \mbox{ even}\\ \BZero, & k \mbox{ odd}\end{array}\right.$$

\par
The $d\times d$ identity matrix is denoted with $\Ib$, and its columns, 
representing the canonical basis in $\RR^d$,  with $\eb_1,\ldots,\eb_d$.
\par
Two filters $\Ab$ and $\Bb$ are said to define a \emph{matrix QMF (Quadrature Mirror Filter) system}, if 
they satisfy the \emph{orthonormality conditions}:
$$\sum_{j\in \ZZ} \Ab(j)^T\Ab(j-2k)=\sum_{j\in \ZZ} \Bb(j)^T\Bb(j-2k)= 2\delta_0(k) \Ib,\quad k \in \ZZ,$$
$$
\sum_{j\in \ZZ} \Ab(j)^T\Bb(j-2k)=\BZero,\quad k \in \ZZ,
$$
which, using the inner product notation, may be written as:
$$
\left<\Ab, T_{2k}\Ab\right> =\left<\Bb, T_{2k}\Bb\right>  = 2 \delta_0(k)\Ib,\quad\left<\Ab, T_{2k}\Bb\right>=\BZero, \quad k\in\ZZ.
$$
For a signal $s\in L^2(\ZZ^d)$ the application of a matrix valued filter is defined in the natural way through
the \emph{convolution} operation:
$$
\Ab \ast \ssb = \sum_{k\in \ZZ} \Ab(\cdot-k)^T\ssb(k). 
$$
A QMF system provides a perfect analysis/reconstruction scheme through the following two equations:
\begin{equation}\label{eq:qmranalysisreconstruction}
\ssb \mapsto
(\ssb_0, \ssb_1) = 
( D(\Ab \ast \ssb), D(\Bb \ast \ssb)),\quad  
D^T(\widetilde\Ab \ast \ssb_0) + D^T(\widetilde\Bb\ast \ssb_1) = \ssb.
\end{equation}
Here we have used the notation $\widetilde\Ab=\Ab^T(-\cdot)$.
\par
Instead of working with matrix valued functions we may also take the following equivalent picture by considering vector valued sequences.
In analogy to the scalar case, a QMF filter system 
is generated by the dilates and translates 
of $2d$  vector-valued function $\ab_1,\ldots, \ab_d$, and $\bb_1,\dots,\bb_d$, satisfying the orthonormality property:
\begin{eqnarray*}
\langle \ab_i,\ab_j(\cdot -k)\rangle &=& 2\delta_0(k)\delta_0(i-j)\\
\langle \bb_i,\bb_j(\cdot -k)\rangle &=& 2\delta_0(k)\delta_0(i-j),\quad i,j=1,\ldots,d,\,k \in \ZZ\\
\langle \ab_i,\bb_j(\cdot -k)\rangle &=& 0.
\end{eqnarray*}
The collection of all these functions is then an orthonormal basis of $L^2(\ZZ)^d$.
We identify these vectors with the previously introduced matrix valued signal by taking them
as their column vectors (analogue for $\bb\leftrightarrow\Bb$):
$$
\Ab=[\ab_1,\ldots,\ab_d]= \sum_{j=1}^d {\ab_j }\, \eb_j^T.
$$
The decomposition and synthesis of a signal in $L^2(\RR^d)$ with respect to this orthonormal basis can then be obtained with the
filters $\Ab$ and $\Bb$  as in (\ref{eq:qmranalysisreconstruction}).
\par
We conclude this section by introducing a crucial feature which  is required to a matrix QMF system to  be  connected to a proper wavelet analysis context.
A QMF system is said to satisfy the \emph{full rank condition} if the  sequence $\Ab$
satisfies:
\begin{equation}\label{eq:fr}
\sum_{j\in \ZZ} \Ab(2j)=\sum_{j\in \ZZ} \Ab(2j+1)=\Ib.
\end{equation}
This, in particular,  implies that the  coefficients of the sequence $\Bb$ sum up to the zero matrix.
\par
It can be shown, that full rank is a necessary (but not sufficient) 
condition for these filters to actually define  a \emph{multichannel multiresolution analysis}  \cite{BacchelliCotroneiSauer02a}, which is a natural extension of the well-known multiresolution
analysis to vector-valued functions. Any (orthogonal) multichannel  MRA is, in fact,  generated 
by two square-integrable $d\times d$-matrix valued functions $\Fb$ and $\Gb$, representing, respectively, the  matrix scaling function and the  multichannel wavelet,   satisfying the 
\emph{matrix two-scale relations}:
$$\Fb=\sum_{j\in \ZZ}\Fb (2\cdot - j)\Ab(j),\quad \Gb=\sum_{j\in \ZZ}\Fb (2\cdot - j)\Bb(j). $$
In such context, the two-scale coefficient sequences $\Ab$, $\Bb$ define a matrix QMF  system and must possess the full rank property. In particular, this is related to the fact that the matrix wavelet $\Gb$ has  \emph{vanishing zeroth moment}:
$$\int_\RR\Gb(x)\,dx=\BZero.$$
We remark that, as in the scalar case, further $p-1$ vanishing moments (usually required in applications)  can be possessed by $\Gb$, that is,
$$\int_\RR x^n \Gb(x)\,dx=\BZero,\quad n=1,\ldots,p-1,$$
 only if the coefficients in the sequence $\Ab$ satisfy the additional \emph{sum rule conditions}:
$$\sum_{k\in \ZZ} (-1)^k k^n \Ab(k)=\BZero,\quad n=1,\ldots,p-1.$$

\section{The construction of matrix QMF systems}
The basic idea of the construction is easiest explained in the well known QMF setting of
a scalar valued signal.
\par
 Let $a$, $b\in L^2(\mathbb{Z})$ be such that
the set of translates by $2\mathbb{Z}$:
\begin{equation}\label{eq:basis}
\{a(\cdot-2m), b(\cdot-2m) : m\in\mathbb{Z}\}
\end{equation}
is  an orthonormal basis of $L^2(\mathbb{Z})$, so that the two sequences form a 
QMF system. Note that the distinction between $a$ and $b$ is somehow arbitrary.
However in actual filtering applications one would like $a$ to be the scaling function filter
and $b$ to be the wavelet filter. This would impose additional constraints (vanishing moments, etc) on
these sequences. 
\par
A trivial example of a QMF system in this purely algebraic setting would be
$$
a = \delta_0,\quad b=\delta_1.
$$
Consider now an orthogonal linear operator $U$ acting in $L^2(\mathbb{Z})$
with the property that it commutes with the shifts by $2$, that is:
$$
[U, T_2]:=UT_2-T_2U=0.
$$
Since $U$ is orthogonal, the image of the basis (\ref{eq:basis})
is again a basis. Since $U$ commutes with the shifts by $2$,
it has the following structure
$$
\{(Ua)(\cdot-2m), (Ub)(\cdot-2m) : m\in\mathbb{Z}\}.
$$
Therefore this basis is again a QMF system.
\par
Vice versa, any QMF system can be obtained from the trivial system in this way.
Indeed, there is exactly one linear operator 
such that: 
$$\delta_{2n}\mapsto a(\cdot-2n), \quad 
\delta_{2n+1}\mapsto b(\cdot-2n),$$
 since any linear operator is defined by its
image of a basis. 
Therefore there is an explicit one-to-one correspondence between QMF filters and
orthogonal operators that commute with translations by 2. 
Note that any such operator is defined through its image of $\delta_0$ and
$\delta_1$. The finite impulse response (FIR) filters correspond exactly to those operators which
leave invariant the non closed linear subspace of compactly  supported
sequences.
\par
Therefore, if we can construct a manifold of such operators, we automatically
have a family of QMF by applying them to any known QMF, in particular the
trivial one \{$\delta_0$, $\delta_1$\}. 
\par
A family of such operators can be obtained in the following way 
(see Fig~\ref{schema}). Pick an
arbitrary $2\times 2$ orthogonal matrix $\Mb\in O(2)$. 
Let the image of $\delta_0$ be its first column put at positions $0$, $1$
extended with zeros on both sides. 
Let 
the image of $\delta_1$ be the second column, put at positions $0$,$1$ and
extended with zeros on both sides:
$$
\delta_0\mapsto M_{1,1}\delta_0 + M_{2,1}\delta_1,\quad
\delta_1\mapsto M_{1,2}\delta_0 + M_{2,2}\delta_1,\quad
$$
Then extend mapping to a linear operator to all of $L^2(\mathbb{Z})$ by
requiring that it
commutes with the shift by two. This defines the operator $U_0(\Mb)$. 
\par
To put it differently: for a sequence $s=(s(n):\, n\in \ZZ)$ $\in$ $L^2(\ZZ)$,
apply the same orthogonal matrix $\Mb$ to the vectors $[s(2k),s(2k+1)]^T$, $k\in \ZZ$. It is plain
that this linear operator conserves the norm and that 
it satisfies the commutation property. 
\par
Now, instead of taking this grouping, we
may also consider the grouping $[s(2k-1),s(2k)]^T$. 
This defines a new operator, $U_1(\Mb)$, which can also be written as
\begin{equation}\label{eq:shiftop}
U_1(\Mb)=T_{-1} \,U_0(\Mb)\,T_1.
\end{equation}
The product of two such operators is again orthogonal and satisfies the
commutation property. Thus, for any finite sequence 
of matrices $\Mb_k\in O(2)$, $k=0,1,\dots K-1$, we have an operator
\begin{equation}\label{eq:operator}
U(\Mb_0,\Mb_1,\dots,\Mb_{K-1})= U_{\epsilon_{K-1}}(\Mb_{K-1})\cdot \ldots \cdot U_{\epsilon_{1}}(\Mb_1)\,U_{\epsilon_{0}}(\Mb_0),
\end{equation}
with $\epsilon_{0},\ldots,\epsilon_{K-1} \in \{0,1\}$, that can applied
 in order to obtain a parameterization of a family of QMF. 
\par
In this case this
parameterization is actually complete. Indeed, for a QMF of length $\leq2N$ the 
outermost blocks are all orthogonal to each other as follows from the orthonormality relations.
Therefore, the spaces spanned by the left block of the $a$ sequence and the left block of the $b$ sequence are collinear.
The same holds for the right blocks. Since they are orthogonal, there is a rotation such that
 the outermost coefficients can be sent to zero and thereby the length 
of the QMF sequence is reduced by $2$. This gives an induction argument for proving the completeness of the parameterization.

\begin{figure*}
\begin{center}
 \includegraphics[width=0.4\textwidth]{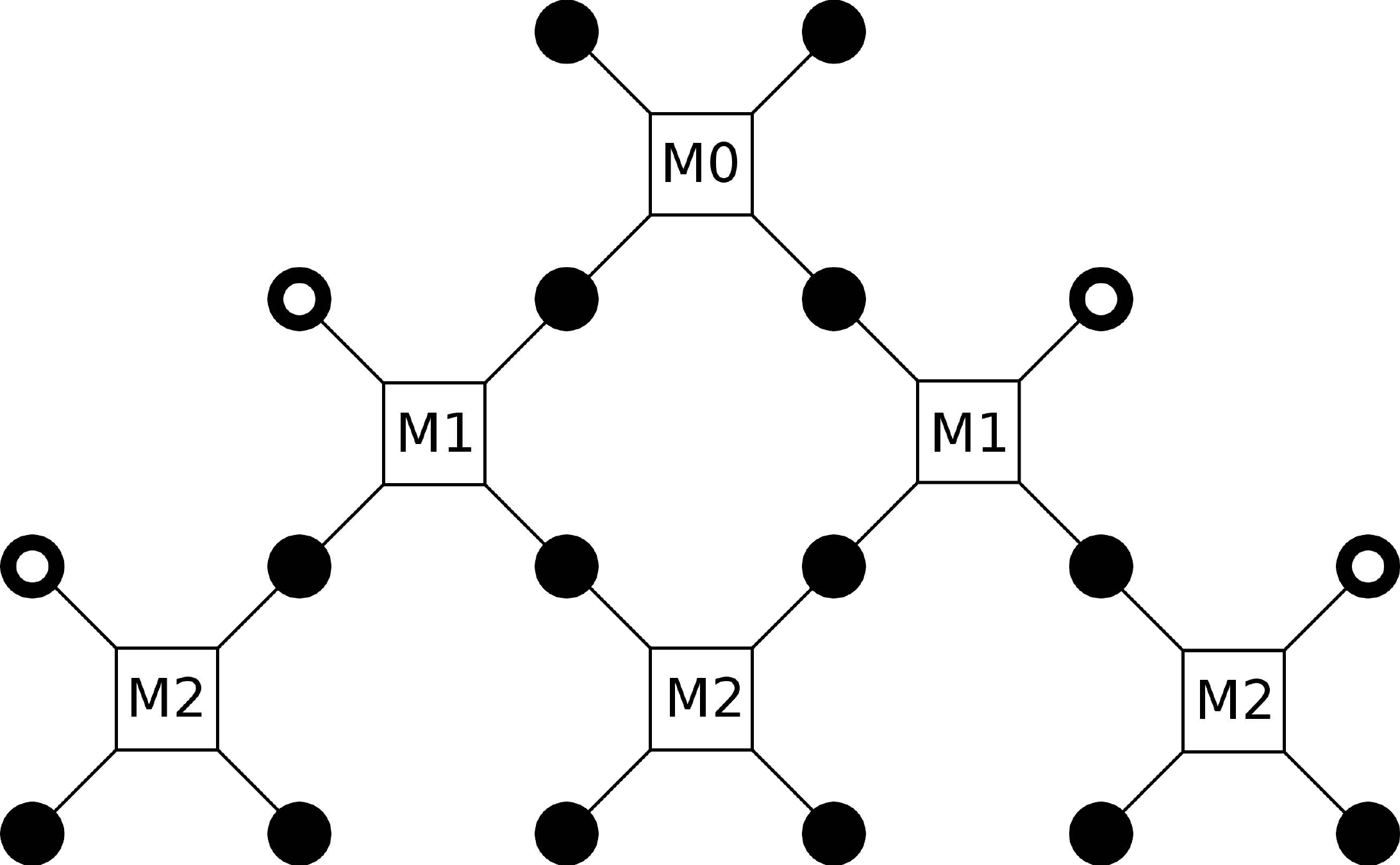}
\caption{Schematic view of the construction. The void circles are $0$.
The black dots correspond to $d$ dimensional vectors.
Note how  the support grows with each application.}\label{schema}
\end{center}
\end{figure*}

A further generalization is possible. Consider any fundamental domain $F$ for
the quotient group $\mathbb{Z}/2\mathbb{Z}$. 
By applying the same matrix operation to the grouping defined by $F$, we obtain
again a linear operator that conserves the norm and commutes with the shifts by $2$. 
Thus, for any sequence of fundamental 
domains and orthogonal matrices,  we may consider the product
$$
U_{F_K}(\Mb_K)\cdot \ldots \cdot U_{F_1}(\Mb_1)\,U_{F_0}(\Mb_0).
$$
It is clear that the same construction can be used to obtain all complex valued filters.
It is enough to replace the orthogonal group and the orthogonal matrices by the unitary analogous.
\par
We  now extend the construction of such operators to the vector case.
As in the previous section, we may identify orthogonal operators in $L^2(\ZZ)^d$ that commute with
$T_2$ and  QMF systems. The trivial vector QMF system is given by
$$
\{\eb_l \delta_0,\quad \eb_l \delta_1,\quad l=1,\dots,d\}.
$$
Pick now a matrix $\Mb\in O(2d)$. We identify $4$ submatrices
$$
\Mb = \left[
\begin{array}{cc}
  \Mb_{1,1}&\Mb_{1,2}\\
  \Mb_{2,1}&\Mb_{2,2}
\end{array}\right].
$$
The orthogonal operator defined in $L^2(\ZZ)^d$ is acting on 
the sequences as  follows:
$$
\eb_l \delta_0 \mapsto \Mb_{1,1} \eb_l \delta_0 + \Mb_{2,1}\eb_l \delta_1,\quad
\eb_l \delta_1 \mapsto \Mb_{1,2} \eb_l \delta_0 + \Mb_{2,2}\eb_l \delta_1.
$$
Again we extend this to all of $L^2(\ZZ)^d$ by requiring that the operator commutes with $T_2$.
\par
In the language of blocks, this operator acts on a vector-valued signal $\ssb\in L^2(\ZZ)^d$ as follows. 
For each $n \in \ZZ$, we identify the two vectors
$\ssb(2n)$ and $\ssb(2n+1)$ with the $2d$-dimensional vector $[\ssb^T(2n)\,\ssb^T(2n+1)]^T$.
The operation on the two vectors in the block is now multiplication with the matrix  $\Mb$.
A second family is obtained by considering, for example, 
the grouping $(2n-1,2n)$ which corresponds to another fundamental domain of $\ZZ/2\ZZ$. 
As before these operations may be combined to obtain
a family of such operators and, a fortiori, a family of matrix filters.
\par
More precisely, in matrix notation, starting from a 
matrix QMF system
\begin{equation}\label{eq:matrbasis}
\{\Ab(\cdot-2m), \Bb(\cdot-2m) : m\in\mathbb{Z}\}
\end{equation}
we apply a matrix 
 $\Mb\in O(2d)$ to the  sequence
$$
\left(\left[ \begin{array}{cc}\Ab(2n)& \Bb(2n)\\
\Ab(2n+1)& \Bb(2n+1)\end{array}\right]: \, n \in \ZZ\right)
$$
or to the sequence
$$
\left(\left[ \begin{array}{cc}\Ab(2n-1)& \Bb(2n-1)\\
\Ab(2n)& \Bb(2n)\end{array}\right]: \, n \in \ZZ\right)
$$
and recursively do the same operation to the transformed sequences.
\par

\par
It is convenient to introduce the  description of the QMF systems in terms of {\em symbols}. 
For this,  we associate with the matrix filter  $\Ab$ the  Laurent series matrix
$$
\Ab(z) = \sum_{k\in \ZZ} \Ab(k) z^{k}, \quad z\in\CC \setminus \{0\},
$$
which in fact is a  Laurent polynomial matrix, if we assume that the filter has finite support.
\par
The so called \emph{subsymbols} are defined through
$$
\Ab_{l}(z)= \sum_{k\in \ZZ} \Ab(2k+l) z^{k},\quad l=0, 1,
$$
and the following relation holds between the symbol and the subsymbols:
$$\Ab(z)=\Ab_0(z^2)+z\Ab_1(z^2).$$
Consider now also the symbol $\Bb(z)$ of the wavelet filter and the respective subsymbols $\Bb_0(z)$, $\Bb_1(z)$ and define the following $2d\times 2d$ polynomial matrix:
$$
\Lb(z) = 
\left[
\begin{array}{cc}
 \Ab_{0}(z) & \Bb_{0}(z)\\
 \Ab_{1}(z) & \Bb_{1}(z)
\end{array}
\right],$$
which is also known as the \emph{polyphase matrix}.
\par
Then the QMF conditions can be written as
 \begin{equation}\label{QMFbis}
[\Lb(z^2)]^H \Lb(z^2)=2\Ib,
\end{equation}
while the full rank requirement is equivalent to 
$$\Lb(1)=\left[ \begin{array}{cc} 
\Ib & \Ib\\
\Ib & -\Ib
\end{array}\right].
$$
Straightforward computations show that the operator $U(\Mb)$ and the translation operator $T_1$ are respectively 
 acting on the symbols  as follows:
 \begin{equation}\label{eq:symbop}
U(\Mb): \Lb(z) \mapsto \Mb \Lb(z),
\end{equation}
\begin{equation}\label{eq:symbshiftop}
T_1: \Lb(z) \mapsto 
\left[
\begin{array}{cc}
 \BZero & \Ib\\
 z^{-1}\Ib & \BZero
\end{array}
\right] \Lb(z),
\end{equation}
so that the operator described in (\ref{eq:shiftop})
 takes the following form in terms of symbols:
$$  \left[
\begin{array}{cc}
 \BZero & z\Ib\\
 \Ib & \BZero
\end{array}
\right] \Mb \left[
\begin{array}{cc}
 \BZero & \Ib\\
z^{-1} \Ib & \BZero
\end{array}
\right].
$$
Given a sequence of matrices $\Mb_k\in O(2d)$, $k=0,1,\dots, K$, by alternately applying (\ref{eq:symbop}) and
(\ref{eq:symbshiftop}) to the matrix $\Lb(z)$, we  obtain the following family of 
QMF symbols:
$$
\left[
\begin{array}{cc}
 \BZero & z\Ib\\
 \Ib & \BZero
\end{array}
\right] 
\Mb_K
\left[
\begin{array}{cc}
 \BZero & \Ib\\
 z^{-1}\Ib & \BZero
\end{array}
\right]   \,
\cdots \, \Mb_2
\left[
\begin{array}{cc}
 \BZero &z \Ib\\
 \Ib & \BZero
\end{array}
\right]
\Mb_1
\left[
\begin{array}{cc}
 \BZero & \Ib\\
 z^{-1}\Ib & \BZero
\end{array}
\right]
\Mb_0\, \Lb(z). 
$$

We conclude the Section by giving an almost characterization of  
the set of matrix QMF that can be obtained using the above construction.
\par
Let $J$ be the smallest interval that contains the support of
all sequences in the QMF system. The even blocks correspond to
grouping the elements with indices in $\{2n,2n+1\}$ 
whereas the odd blocks are the one with $\{2n-1,2n\}$.
The outermost blocks of an even or odd covering of $\ZZ$ 
are the left most and right most that intersect the support. 
The dimension of a block is defined to be the dimension of the
vector space generated by the $2d$ vectors with indices in the block. 

\begin{theorem}
All QMF systems that can be obtained through this construction
have the dimension of their outermost blocks $\leq d$.
Vice versa all matrix QMF system of finite support length 
for which the dimension of outermost block is equal to $d$
can be obtained through this construction. 
\end{theorem}

\begin{proof}
For length equal to $2$ nothing has to be proved.

That the dimensions of the outermost blocks are at most $d$ is clear, since it is obtained
by applying a linear map to a vector space of dimension at most $d$.

Vice versa, suppose that the outermost blocks for either an even or odd covering 
are each  $d$ dimensional. 
By orthogonality of the basic vector sequences, all the extremal blocks on the left are
orthogonal to all the extremal blocks on the right. Since the dimension
of each of these blocks is $d$ it is possible to find an orthogonal matrix 
such that the left block is mapped into a block having only
zeros to the left, and the right block is mapped into a blocking
having zeros to the right. Therefore the length of the filters is reduced by
$2$. 
Now, the dimension of each of the new extremal blocks is at least $d$. 
Indeed, consider the new right end block. The outermost parts, which correspond to the
lower $d$ components of the $2d$ dimensional block vector are $d$ dimensional since they
correspond to the orthogonal image of the previous $d$-dimensional right-most block. 
The same holds for the new left-most block. 
Since the new extremal blocks are again orthogonal, their dimension is equal to $d$.

This concludes an induction argument.
\end{proof}

\section{Examples}
From the results presented in the previous section, it turns out that the actual challenge in the construction of matrix filters is the choice of the orthogonal matrices $\Mb_k$ in the $d(2d-1)$-dimensional Lie group $O(2d)$.
\par
To give  examples of such construction, we  restrict ourselves to the case of the  special orthogonal group $SO(2d)$ of rotations.

\subsection{Construction of a two-channel filter bank}
 In this first example, we let $d=2$ and consider a representation of  $SO(4)$ in terms of  products of Givens rotation matrices $ \Gb(\ell,m,\theta)$, with $1\le\ell<m\le 4$,  whose only  non-zero elements $g_{ij}$ are given by:
 $$g_{kk}=1, \, k\ne \ell,m\quad 
 g_{\ell \ell}=g_{m m}=\cos\theta,\quad  g_{\ell m}=-g_{m \ell}=\sin\theta.$$
Let us take, as initial set of QMF filters, the  Haar matrix filter bank, 
where:
$$\Ab(0)=\Ab(1)= \Bb(0)=-\Bb(1)=\Ib.$$
We first consider a  6-parameter rotation transformation on the the following sequence of $2\times 2$ matrices:
$$\left( \cdots, \, \left[ \begin{array}{cc}\BZero & \BZero \\ \BZero  & \BZero  \end{array}\right],\, \left[ \begin{array}{cc}\BZero & \BZero\\ \Ib & \Ib \end{array} 
\right],\,
\left[ \begin{array}{cc}\Ib & -\Ib \\ \BZero  & \BZero  \end{array}\right],\,
\left[ \begin{array}{cc}\BZero & \BZero \\ \BZero  & \BZero  \end{array}\right]
\cdots \right)$$
This transformation produces two  4-length sequences $\widetilde \Ab(k), \widetilde \Bb(k)$, $k=0,\ldots, 3$ depending on 6 parameters. They can then be arranged as
$$\left( \cdots, \, \left[ \begin{array}{cc}\BZero & \BZero \\ \BZero  & \BZero  \end{array}\right],\,
 \left[ \begin{array}{cc}\widetilde \Ab(0) & \widetilde \Bb(0)\\ \widetilde \Ab(1) & \widetilde \Bb(1) \end{array} 
\right],\,
 \left[ \begin{array}{cc}\widetilde \Ab(2) & \widetilde \Bb(2)\\ \widetilde \Ab(3) & \widetilde \Bb(3) \end{array} 
\right],\,
\left[ \begin{array}{cc}\BZero & \BZero \\ \BZero  & \BZero  \end{array}\right]
\cdots \right)$$
and transformed through another 6-parameter rotation.
The resulting filter bank consists of two 6-length sequences $ \widehat \Ab(k), \widehat \Bb(k)$, $k=0,\ldots, 5$ depending on 12 parameters. These degrees of freedom can be  exploited to impose additional constraints. In particular, we can require the full rank constraint (\ref{eq:fr}) and the second order  sum rule condition (which produces an additional vanishing moment on the wavelet filter):
$$\sum_{k\in\ZZ}(-1)^k k \Ab(k)=\BZero.$$
As an example, let $S(\theta)$, with $\theta=[\theta_1,\theta_2,\theta_3,\theta_4,\theta_5,\theta_6]$, be the $6$-parameter rotation matrix obtained as the following product of Givens matrices:
$$S(\theta)=\Gb(1,2,\theta_4)\Gb(3,4,\theta_3)\Gb(2,3,\theta_2)\Gb(1,4,\theta_1)\Gb(1,3,\theta_6)\Gb(2,4,\theta_5),$$
explicitly given by
\begin{center}
{$\left[ \begin {array}{cccc} c_{6}c_{1}c_{4}-s_{6}s_{2}s_{4
}&-s_{5}s_{1}c_{4}+c_{5}c_{2}s_{4}&s_{6}c_{1}c_{4}+
c_{6}s_{2}s_{4}&c_{5}s_{1}c_{4}+s_{5}c_{2}s_{4}
\\\noalign{\medskip}-c_{6}c_{1}s_{4}-s_{6}s_{2}c_{4}&s_{5
}s_{1}s_{4}+c_{5}c_{2}c_{4}&-s_{6}c_{1}s_{4}+c_{6}s
_{2}c_{4}&-c_{5}s_{1}s_{4}+s_{5}c_{2}c_{4}
\\\noalign{\medskip}-s_{6}c_{2}c_{3}-c_{6}s_{1}s_{3}&-c_{
5}s_{2}c_{3}-s_{5}c_{1}s_{3}&c_{6}c_{2}c_{3}-s_{6}s
_{1}s_{3}&-s_{5}s_{2}c_{3}+c_{5}c_{1}s_{3}
\\\noalign{\medskip}s_{6}c_{2}s_{3}-c_{6}s_{1}c_{3}&c_{5}
s_{2}s_{3}-s_{5}c_{1}c_{3}&-c_{6}c_{2}s_{3}-s_{6}s_
{1}c_{3}&s_{5}s_{2}s_{3}+c_{5}c_{1}c_{3}\end {array}
 \right] 
$}
\end{center}
with
$$ c_i=\cos(\theta_i),\,
s_i=\sin(\theta_i), \, i=1,\ldots,6.$$
Let now  
$\Mb_0=S(\phi)$,  $\Mb_1=S(\psi)$ be the rotations (depending on the 12 free parameters
$\phi_1,\ldots,\phi_6,\psi_1,\ldots,\psi_6$) respectively used for the first and the second transformation.
\par
The 12 degrees of freedom  can  be fully exploited to impose the 12 full rank/sum rule conditions.
The solution is given in terms of the following parameter vectors:
\begin{center}
{\small$ \phi=[-1.530817,  -2.054355,  -2.642328, 0.495166,1.413293,1.728299]$},
{\small $\psi=[ -2.345058,  2.382453,  -1.422064,  -1.696487,  1.165227,  -1.439620]$,}
\end{center}
which give rise to filters convergent to the matrix scaling function and wavelet illustrated in Fig.~\ref{fig:examplescalwav}.

%
%
% we  consider the following rotation matrices:
%$$M_0=S(\phi),\quad M_1=S(\psi)$$
%where
%$S(\theta)$ is obtained as the following product of Givens matrices
%
%
%\begin{center}
%{ $S(\theta)=\left[ \begin {array}{cccc} c_{6}c_{1}c_{4}-s_{6}s_{2}s_{4
%}&-s_{5}s_{1}c_{4}+c_{5}c_{2}s_{4}&s_{6}c_{1}c_{4}+
%c_{6}s_{2}s_{4}&c_{5}s_{1}c_{4}+s_{5}c_{2}s_{4}
%\\\noalign{\medskip}-c_{6}c_{1}s_{4}-s_{6}s_{2}c_{4}&s_{5
%}s_{1}s_{4}+c_{5}c_{2}c_{4}&-s_{6}c_{1}s_{4}+c_{6}s
%_{2}c_{4}&-c_{5}s_{1}s_{4}+s_{5}c_{2}c_{4}
%\\\noalign{\medskip}-s_{6}c_{2}c_{3}-c_{6}s_{1}s_{3}&-c_{
%5}s_{2}c_{3}-s_{5}c_{1}s_{3}&c_{6}c_{2}c_{3}-s_{6}s
%_{1}s_{3}&-s_{5}s_{2}c_{3}+c_{5}c_{1}s_{3}
%\\\noalign{\medskip}s_{6}c_{2}s_{3}-c_{6}s_{1}c_{3}&c_{5}
%s_{2}s_{3}-s_{5}c_{1}c_{3}&-c_{6}c_{2}s_{3}-s_{6}s_
%{1}c_{3}&s_{5}s_{2}s_{3}+c_{5}c_{1}c_{3}\end {array}
% \right] 
%$}
%\end{center}
%with
%$${ \theta=[\theta_1,\theta_2,\theta_3,\theta_4,\theta_5,\theta_6], \, \mbox{ and }\, c_i=\cos(\theta_i),\,
%s_i=\sin(\theta_i), \, i=1,\ldots,6}$$
%and fully exploit the 12 degrees of freedom to impose the 12 full rank/sum rule conditions.
%The solution is given in terms of the following parameters:
%\begin{center}
%{\small$ \phi=[-1.530817,  -2.054355,  -2.642328, 0.495166,1.413293,1.728299]$},
%{\small $\psi=[ -2.345058,  2.382453,  -1.422064,  -1.696487,  1.165227,  -1.439620]$,}
%\end{center}
%which give rise to filters convergent to the matrix scaling function and wavelet illustrated in Fig.~\ref{fig:examplescalwav}.

\begin{figure}[h]
\begin{center}
\fbox{\includegraphics[width=5.5cm]{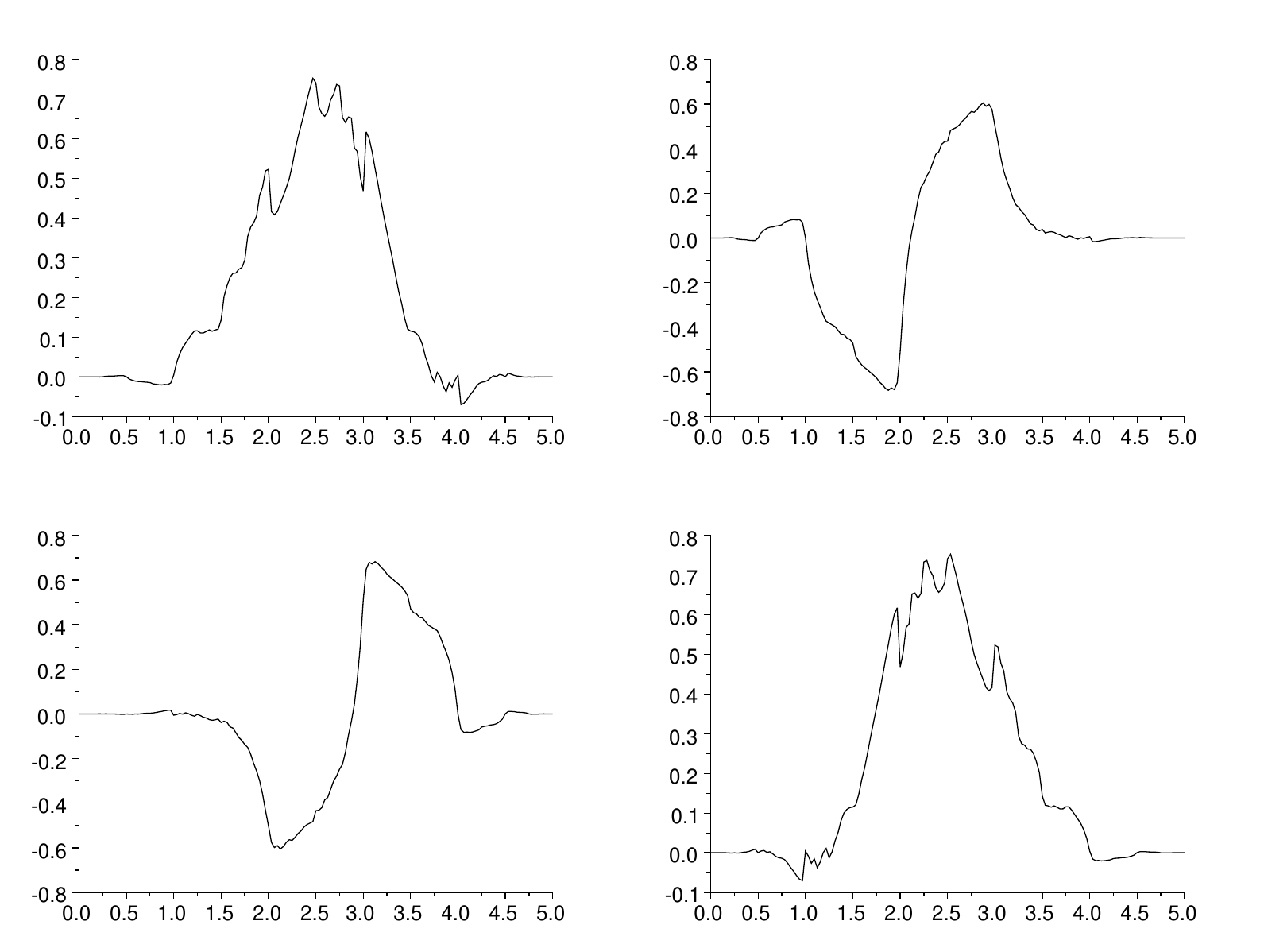}}\ 
\fbox{\includegraphics[width=5.5cm]{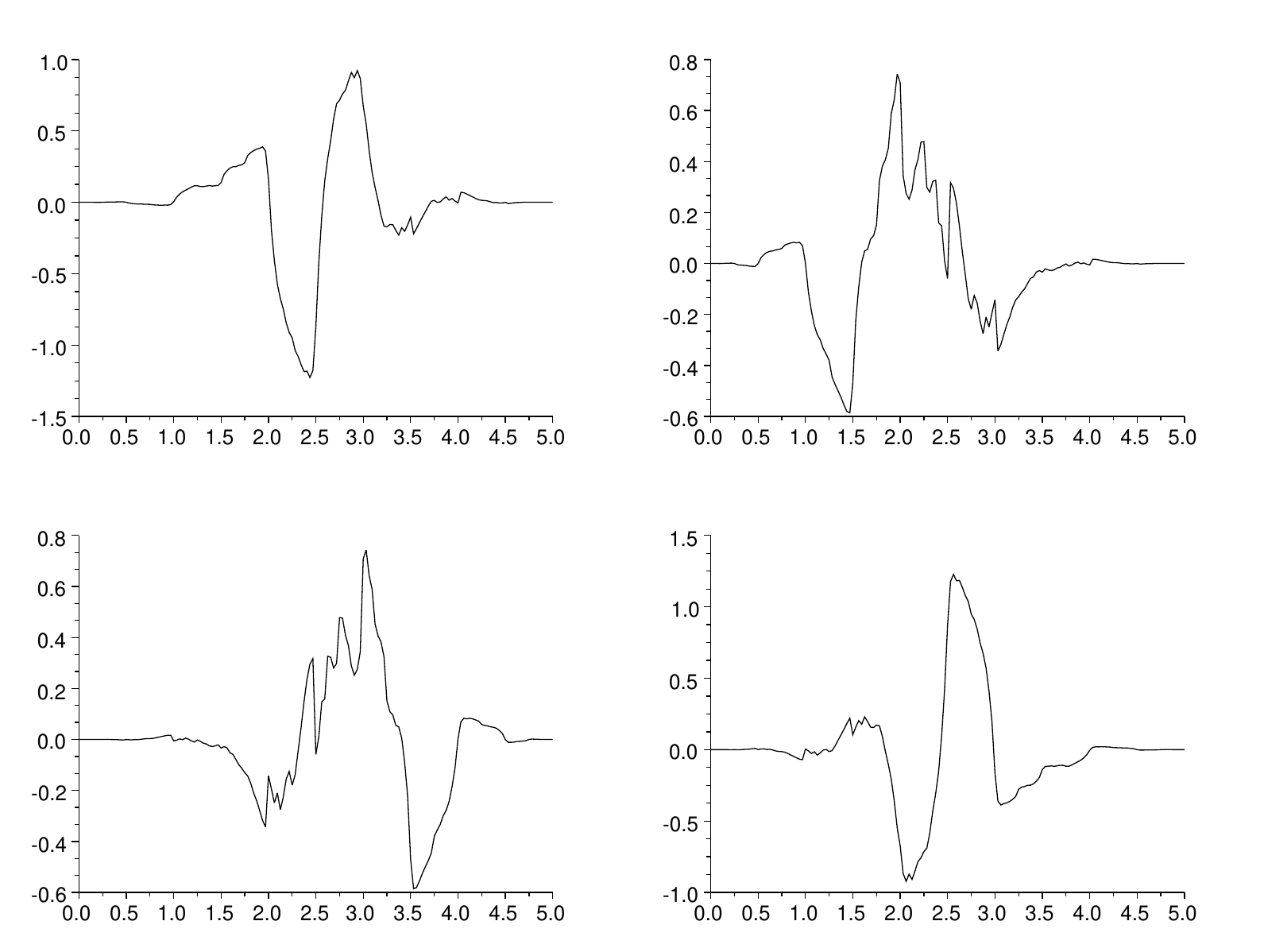}}
\caption{The four components of the 2-channel scaling function (left) and wavelet (right) associated to the  parameters obtained in the example in Section 4.1.}\label{fig:examplescalwav}
\end{center}
\end{figure}

\subsection{Construction of a 4-parameter family}
In this subsection we propose a special construction of a family of filters with $d=2$ and support length $\leq4$.
 \par
Instead of working with the group $SO(4)$ itself, we consider a representation of its elements in terms of infinitesimal generators, given by the following elements of its Lie algebra:
$$
\Xb_{\alpha}=\Eb^{i,j}-\Eb^{j,i},\quad 1\leq i<j\leq4,\quad \alpha=\alpha(i,j)
$$
where $\Eb^{i,j}$ is the matrix whose elements are all zero except the element at the $i$-th row and $j$-th column, which is equal to one.
\par
We take the following ordering of the indices $(i,j)$ for the corresponding subscript $\alpha=\alpha(i,j)$, $1\le\alpha\le 6$:
$$
(1,2),\ (1,3),\ (1,4),\ (2,3),\ (2,4),\ (3,4).
$$
By taking exponentials of linear combinations, we obtain the  elements of the group  as 
$$
\Mb(\xib)= \exp\left(\sum_{\alpha=1}^6 \xi_\alpha \, \Xb_\alpha\right),
$$
where we have set $\xib=[\xi_1,\xi_2,\dots,\xi_6]^T\in \RR^6$.
\par
Let us now take two rotations parameterized by $\xib$ and $\xib^\prime$. 
Starting from the Haar system, 
whose polyphase matrix is given by:
$$
\Hb(z)= \left[ \begin{array}{cc}
        \Ib & \Ib\\
	\Ib & -\Ib
       \end{array}
\right],
$$
our construction gives us a $12$ parameter family of orthogonal 
filters with parameters $[\xib^T,{\xib^\prime}^T]^T\in\RR^{12}$. 
The symbols of these filters are given by
$$
\Lb(\xib,\xib^\prime; z) = 
\left[ \begin{array}{cc}
        \BZero & z\Ib\\
	\Ib & \BZero
       \end{array}
\right]\,
\Mb(\xib^\prime)\,
\left[ \begin{array}{cc}
        \BZero & \Ib\\
	z^{-1}\Ib & \BZero
       \end{array}
\right]\,
\Mb(\xib)
\Hb(z).
$$
Not all of them are different however. And not all of them are full rank either.
Indeed, consider two matrices in $\Ub,\Vb\in SO(2)$. The following choice of matrices in $SO(4)$ will all yield the
Haar system
$$
\Mb=\left[ \begin{array}{cc}
        \Ub & \BZero\\
	\BZero & \Vb
       \end{array}
\right],\quad
\Mb^\prime=\left[ \begin{array}{cc}
        \Vb^T & \BZero\\
	\BZero & \Ub^T
       \end{array}
\right].
$$
For small perturbations of the Haar system a linear analysis of the space of filters obtained that way may be performed.
For this we write
$$
\Lb(\xib,\xib^\prime;z)= z^{-1} \Lb_{-1}(\xib,\xib^\prime) 
+ z^{0} \Lb_{0}(\xib,\xib^\prime) + z^{1} \Lb_{1}(\xib,\xib^\prime). 
$$
This defines an embedding of the space of such filters into $\mathbb{R}^{3 \cdot 16}$ by taking the
coefficients of the matrices $\Lb_{k}$, $k=-1,0,1$.
\par
 Consider now the Jacobian of the
nonlinear mapping $(\xib,\xib^\prime)\mapsto \Lb(\xib,\xib^\prime;\cdot)\to\mathbb{R}^{3 \cdot 16}$ at the origin
and  call this linearization $\Gamma$. 
As elementary linear algebra shows, the rank is $10$ and hence the defect is $2$. A basis of its kernel is given by:
$$
\mbox{kern}\, \Gamma = \mbox{span}\{\gammab_1,\gammab_2\}
$$
with
$$
\gammab_1=[-1,0,0,0,0,0,0,0,0,0,0,1]^T
\quad\gammab_2=[0,0,0,0,0,-1,1,0,0,0,0,0]^T
$$
So at least locally we obtain a parameterization in terms of $\RR^{12}/\mbox{kern}\,\Gamma$. 
\par
We take as complementary system the following $10$ vectors written as matrix of column vectors
$$
\begin{array}{|r|r|r|r|r|r|r|r|r|r|}
 \hline
 1 &0 &0 &0 &0 &0 &0 &0 &0 &0\\
 0 &1 &0 &0 &0 &0 &0 &0 &0 &0\\
 0 &0 &1 &0 &0 &0 &0 &0 &0 &0\\
 0 &0 &0 &1 &0 &0 &0 &0 &0 &0\\
 0 &0 &0 &0 &1 &0 &0 &0 &0 &0\\
 0 &0 &0 &0 &1 &0 &0 &0 &0 &0\\
 \hline
 0 &0 &0 &0 &0 &1 &0 &0 &0 &0\\
 0 &0 &0 &0 &0 &0 &1 &0 &0 &0\\
 0 &0 &0 &0 &0 &0 &0 &1 &0 &0\\
 0 &0 &0 &0 &0 &0 &0 &0 &1 &0\\
 0 &0 &0 &0 &0 &0 &0 &0 &0 &1\\
 1 &0 &0 &0 &0 &0 &0 &0 &0 &0\\
 \hline
\end{array}
$$
Note that the two directions $\gammab_1$ and $\gammab_2$ correspond exactly to the above noted
family of matrices which do not alter the Haar wavelets. 
\par
We now want to explore those directions that yield one parameter subgroups of full rank filters. 
Only a sufficient condition will be given. The condition of full rank requires that
$$
\Lb(\xib,\xib^\prime;1)=\Hb(1).
$$
We are unable to solve this system in general. However,
restricting our research to special solutions of the form $\Lb(t\xib,t\xib^\prime;z)$, $t\in\mathbb{R}$, we can construct a
finite dimensional submanifold of full rank filters spanned by the one-dimensional
abelian subgroup of rotations in $SO(4)\times SO(4)$  with their generators forming a  sub-vector space (of its Lie algebra). 
\par
Since for $t=0$ the system is full rank, it is enough to require that 
the derivative with respect to $t$ is $0$ for all $t$:
$$
\frac{d}{dt} \Lb(t\xib,t\xib^\prime;1)=0.
$$
Substituting this and deriving with
respect to $t$ yields the following matrix system of equations:

\begin{eqnarray*}
&&\left[ 
    \begin{array}{cc}
        \BZero & \Ib\\
	\Ib & \BZero
     \end{array}
\right]\, 
\Mb(t\xib)
\left( (\sum_\alpha \xib_\alpha \Xb_\alpha)\,
\left[ \begin{array}{cc}
        \BZero & \Ib\\
	\Ib & \BZero
       \end{array}
\right]\right.\\
&\quad&\quad\quad +
\left.
\left[ \begin{array}{cc}
        \BZero & \Ib\\
	\Ib & \BZero
       \end{array}
\right]
(\sum_\alpha \xib_\alpha^\prime \Xb_\alpha)
\right) \cdot\Mb(t\xib^\prime)\,
\left[ \begin{array}{cc}
        \Ib & \Ib\\
	\Ib & -\Ib
       \end{array}
\right] =
\left[ \begin{array}{cc}
        \BZero & \BZero\\
	\BZero & \BZero
       \end{array}
\right] .
\end{eqnarray*}

A sufficient condition is given by the inner parenthesis to vanish. This yields a linear system of $16$ 
equations for the $12$ parameters $\xib$, $\xib^\prime$. Due to linear dependencies among the equations
the system has a $6$ dimensional solution space which is spanned by the following $6$ vectors
\begin{center}
\begin{tabular}{|r|r|r|r|r|r|} 
0 &-1 &0 &0 &0 &0\\
0 & 0 &0 &1 &0 &0\\
0 & 0 &0 &0 &1 &0\\
1 & 0 &0 &0 &0 &0\\
0 & 0 &1 &0 &0 &0\\
0 & 0 &0 &0 &0 &-1\\
\hline
0 & 0 &0 &0 &0 &1\\
0 & 0 &0 &1 &0 &0\\
1 & 0 &0 &0 &0 &0\\
0 & 0 &0 &0 &1 &0\\
0 & 0 &1 &0 &0 &0\\
0 & 1 &0 &0 &0 &0
\end{tabular}
\end{center}
The family of full rank filter banks that are generated by rotations associated with the 
corresponding infinitesimal rotations however is only $4$ dimensional.  Indeed, as may be verified through elementary linear algebra, the 
intersection of the linear span  of these vectors and the $10$ dimensional 
space of directions of local filter systems is a four dimensional vector space
spanned by the following column vectors, which,
for the convenience of the reader, we have splitted  into the respective part for  $\xib^\prime$ and $\xib$ 
corresponding to the outer and inner rotation matrices:
\begin{center}
\begin{tabular}{|r|r|r|r|}
\multicolumn{4}{|c|}{$\xib^\prime$}\\
\hline
0 &0  &0 &0\\
0 &0  &0 &1\\
1 &0  &0 &0\\
0 &0  &1 &0\\
0 &1  &0 &0\\
0 &1  &0 &0\\
\hline
\end{tabular}
\quad
\begin{tabular}{|r|r|r|r|}
\multicolumn{4}{|c|}{$\xib$}\\
\hline
0 &-1 &0 &0\\
0 &0  &0 &1\\
0 &0  &1 &0\\
1 &0  &0 &0\\
0 &1  &0 &0\\
0 &0  &0 &0\\
\hline
\end{tabular}
\end{center}
If we denote with $\Gb_{\xi'},\Gb_{\xi}$ the matrices of dimension $6\times4$ corresponding to the above column vectors and let $\pb=[\eta, \theta, \omega,\zeta]^T$ be the vector containing the free parameters, then the final form of our parameterization is expressed in terms of the following vectors:
$$\xib^\prime=\Gb_{\xi'}\,\pb=\left[ \begin {array}{cccccc} 0&\zeta&\eta&\omega&\theta&\theta
\end {array} \right], \quad \xib= \Gb_{\xi}\,\pb=\left[ \begin {array}{cccccc} -\theta&\zeta&\omega&\eta&\theta&0
\end {array} \right] 
$$
An expression for the filters $\Ab(z)$ and $\Bb(z)$ depending on the above four parameters is however difficult to give, due to the presence of the exponentials of the matrix sums $\sum_\alpha \xi_\alpha \Xb_\alpha$, $\sum_\alpha \xi^\prime_\alpha \Xb_\alpha$. 
\par
Strong simplifications occur when we set all but one direction parameters to zero. All of the corresponding 
 one-parameter filter families can thus be explicitly given.
 \par
 The family corresponding to the only non-zero direction parameter $\eta$ is given in terms of the following symbols:
\begin{equation}\Ab(z)=(z+1)\left[
 \begin{array}{cc}
 a^2z+(1-a^2) & (1-z)a\sqrt{1-a^2}\\
 z(1-z) a\sqrt{1-a^2}  &  (1-a^2)z^2+a^2z
 \end{array}\right]
 \label{eq:firstA}
 \end{equation}
 
 \begin{equation}
 \Bb(z)=(1-z)\left[
 \begin{array}{cc}
 a^2z-(1-a^2) & (1+z) a\sqrt{1-a^2}\\
 -z(z+1) a\sqrt{1-a^2}  &  (1-a^2)z^2-a^2z
 \end{array}\right] \label{eq:firstB}
\end{equation}
 where we have set $a=\cos \eta$ (see  Fig. 3).
 \par
 Observe that $\Bb(z)$ is obtained as $\Pb \Ab(-z)\Pb$, where
 $\Pb=\left[ \begin{array}{cc} 1 &0 \\ 0&-1\end{array} \right]$.
 \par
 As to the convergence of such filters to matrix scaling function/wavelets, observe that the autocorrelation symbol
 $\Cb(z)=\frac 12 \Ab^H (z) \Ab(z)$
 is given by $\Cb(z)= \frac{(z+1)^2}{2z}\Ib$. The positive definiteness of $\Cb(z)$ assures the convergence of $\Ab(z)$ to an $L^2(\RR)^d$ orthogonal matrix scaling function as proved in \cite{ContiCotroneiSauer07S}.
 \par
 Note that essentially the same family is attained taking $\omega$ as unique non zero parameter. In such case only the role of $\xi$ and $\xi^\prime$ exchange.
 
 \begin{figure}[h]
\begin{center}
\fbox{\includegraphics[width=5.5cm]{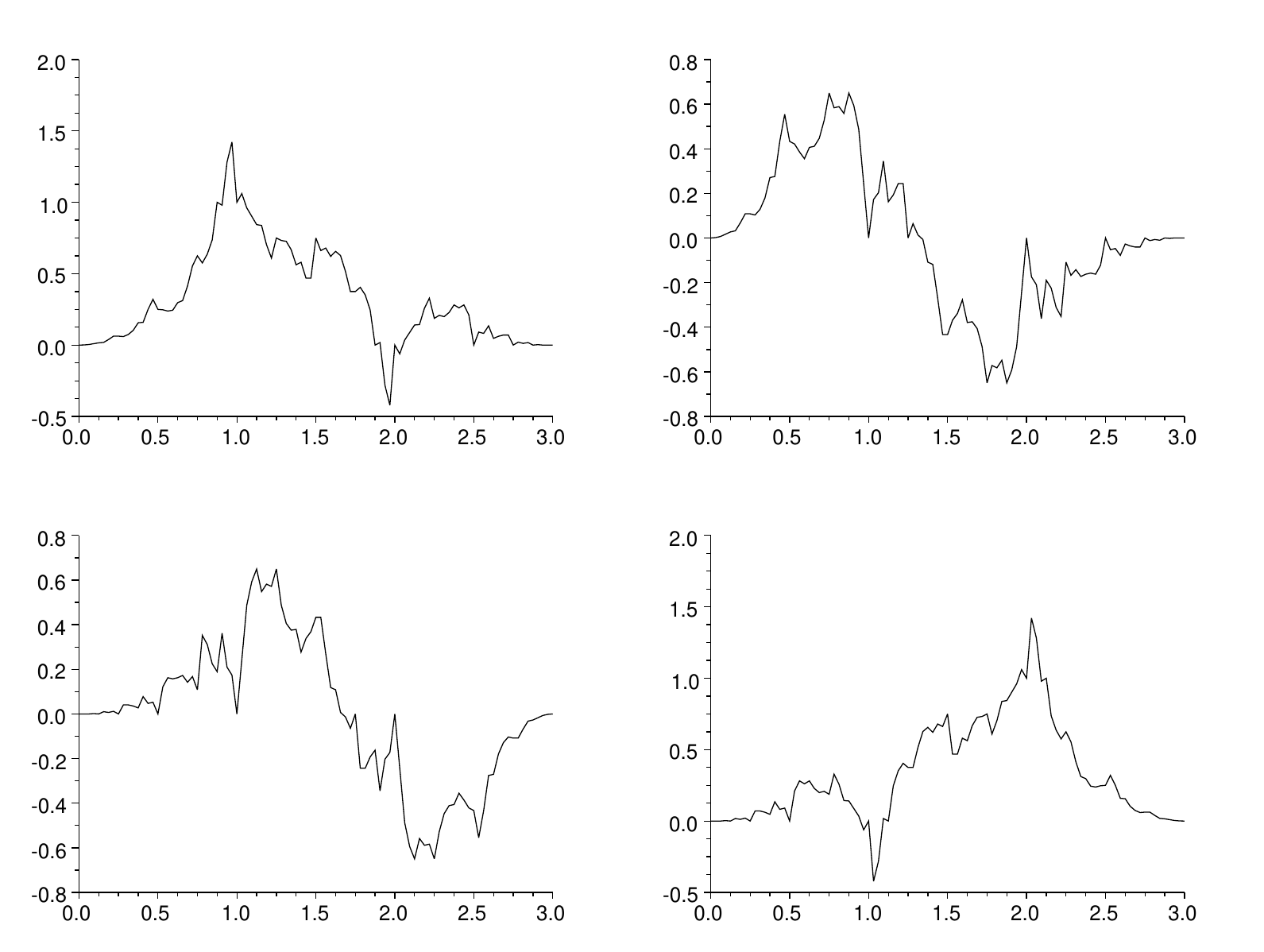}}\ 
\fbox{\includegraphics[width=5.5cm]{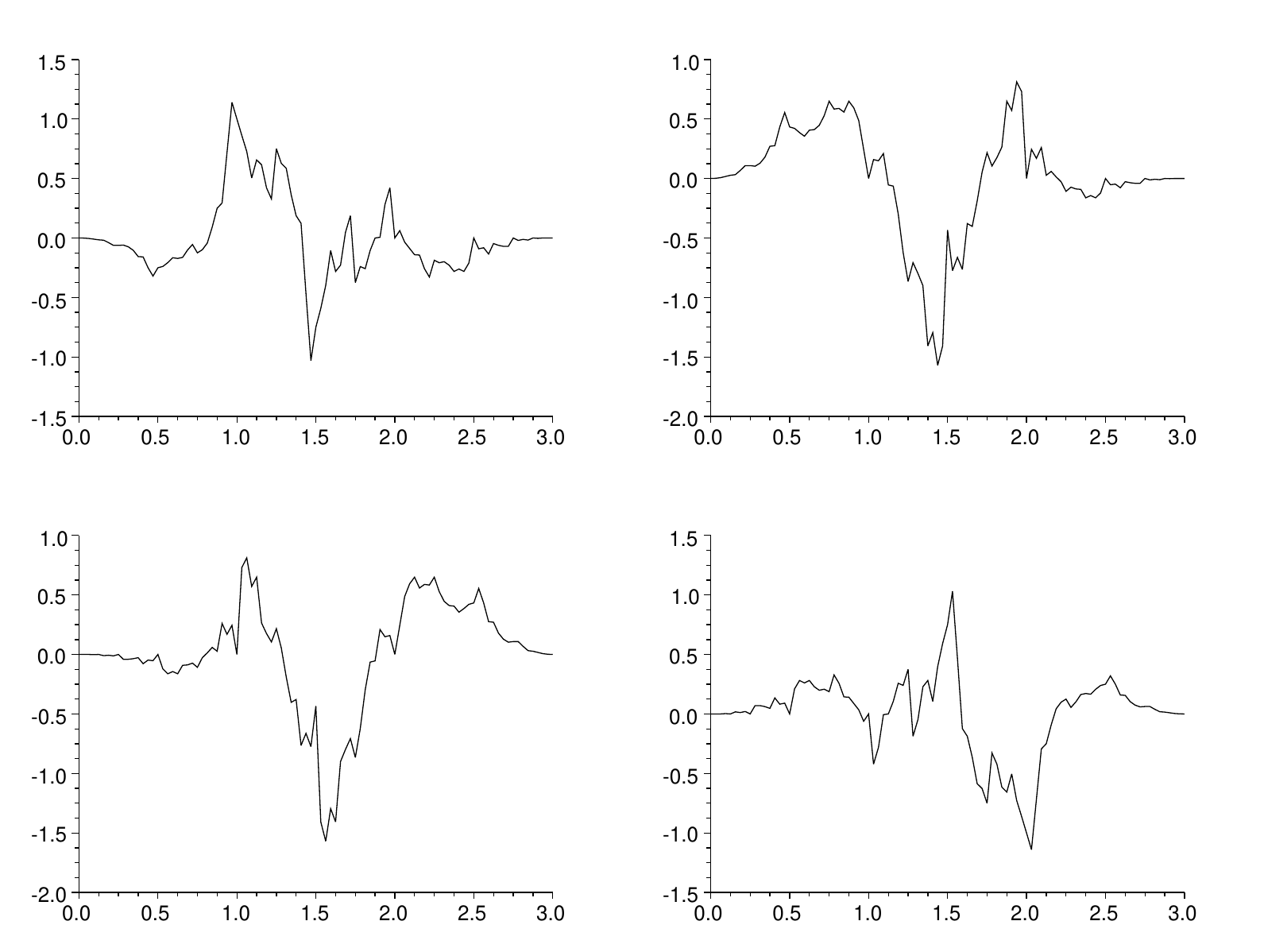}}
\end{center}
\caption{The four components of the 2-channel scaling function (left) and wavelet (right) associated to the symbols (\ref{eq:firstA}), (\ref{eq:firstB}) with $a=\cos(\pi/6)$}\label{fig:fig3}
\end{figure}

 \par
 Let us now consider the case of a second family, obtained by letting $\theta$ be the only non zero  parameter 
 in $\xi, \xi^\prime$. By setting $a=\cos(\sqrt{2}\,\theta$),  $\Ab(z)=\frac 14 \,(z+1)\tilde \Ab(z)$, $\Bb(z)=\frac 14\,(1-z) \tilde \Bb(z)$ the elements of  $\tilde \Ab(z)$ and $\tilde \Bb(z)$ read as:
\begin{eqnarray}\label{eq:secondA}
\nonumber
 \tilde  \Ab_{11}(z)&=&(a-1)^2 z^2 +(3+2a-a^2)z\\\nonumber
\tilde \Ab_{12}(z)&=&z(1-z)(1-a)\left(1+a-\sqrt{2(1-a^2)}\right)\\
\tilde \Ab_{21}(z)&=&(1-z)(1-a)\left(1+a-z\sqrt{2(1-a^2)}\right)\\\nonumber
\tilde \Ab_{22}(z)&=&(a+1)\left(2-2a-\sqrt{2(1-a^2)}\right)z^2+(3a^2+2a-1)z\\\nonumber
&&\quad +\left(3-a^2-2a+(a+1)\sqrt{2(1-a^2)}\right)
       \end{eqnarray}
\begin{eqnarray}\label{eq:secondB}
\nonumber
 \tilde  \Bb_{11}(z)&=&(a-1)^2 z^2 +(a^2-2a-3)z\\\nonumber
\tilde \Bb_{12}(z)&=&z(z+1)(1-a)\left(1+a+\sqrt{2(1-a^2)}\right)\\
\tilde \Bb_{21}(z)&=&(z+1)(1-a)\left(-1-a+z\sqrt{2(1-a^2)}\right)\\\nonumber
\tilde \Bb_{22}(z)&=&(a+1)\left(2-2a+\sqrt{2(1-a^2)}\right)z^2+(1-3a^2-2a)z\\\nonumber
&&\quad +\left(3-a^2-2a-(a+1)\sqrt{2(1-a^2)}\right)
\end{eqnarray}

\par
Elementary computations show that also in this case the corresponding autocorrelation symbol 
$\Cb(z)$ is positive definite, which indicates the convergence of the filters to orthogonal matrix scaling functions/wavelets for any value of the parameter $a$ in $[-1,1]$ (see Fig. 4).

\begin{figure}[h]
\begin{center}
\fbox{\includegraphics[width=5.5cm]{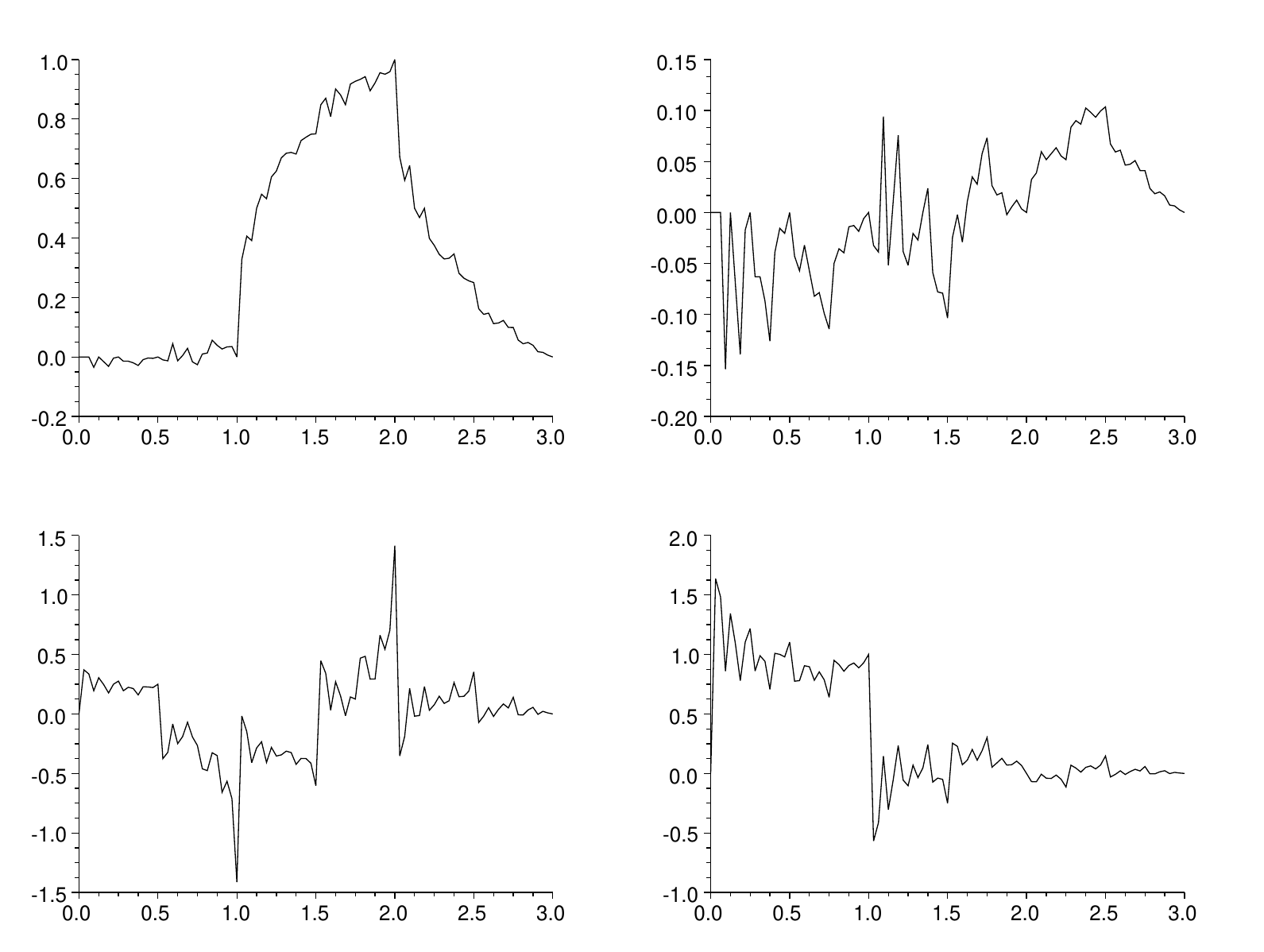}}\ 
\fbox{\includegraphics[width=5.5cm]{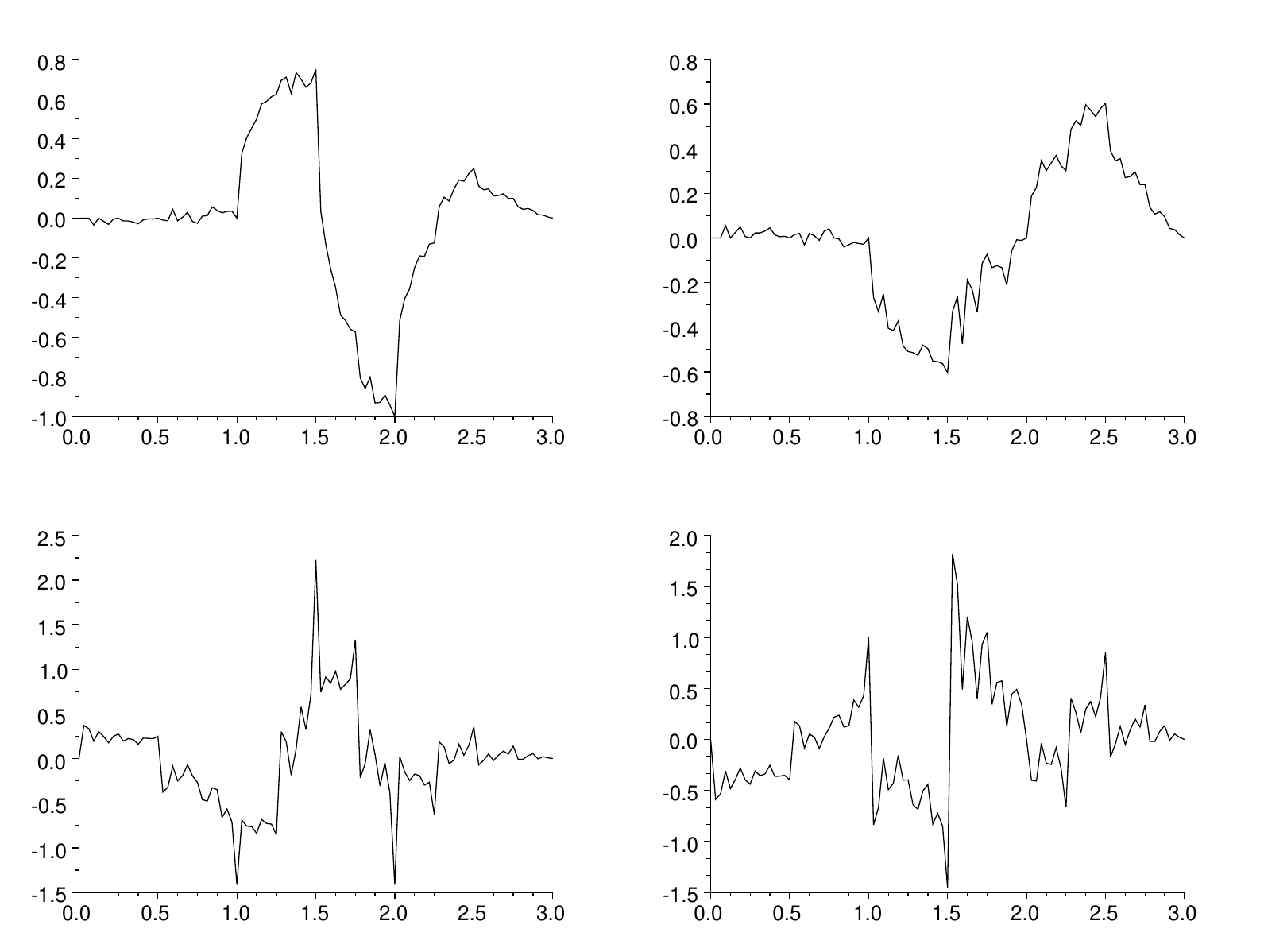}}
\end{center}
\caption{The four components of the 2-channel scaling function (left) and wavelet (right) associated to the symbols (\ref{eq:secondA}), (\ref{eq:secondB}) with $a=0$}\label{fig:fig4}
\end{figure}

\par
The last one parameter family is that obtained by setting all parameters except $\zeta$ to zero. However, this is a trivial case, which means that, setting $a=\sin(\zeta)$,  the symbols $\Ab(z)$ and $\Bb(z)$ reduce to the diagonal forms:
{ \begin{equation}\Ab(z)=(z+1)\left[
 \begin{array}{cc}
 (a^2-a\sqrt{1-a^2})z +1-2a^2 +(a^2+a\sqrt{1-a^2})z^{-1}& 0\\
 0 &  1
 \end{array}\right]\end{equation}}
 { \begin{equation}\Bb(z)=(1-z)\left[
 \begin{array}{cc}
 (a^2+a\sqrt{1-a^2})z +2a^2-1 +(a^2-a\sqrt{1-a^2})z^{-1}& 0\\
 0 &  1
 \end{array}\right]\end{equation}}
and the scheme thus  reduces to a combination of two scalar schemes, namely the Haar scalar system plus a parameterized version of the Daubechies system (attained for $a=\sqrt{3}/{2}$). This last system is convergent for any $a\in [-1,1]\setminus \{0\}$.
 
\par
We conclude the section presenting an example of a 2 parameter family obtained by considering the situation 
$\theta=\zeta=0$. In this case, the argument of the exponential function is a sum of two commuting matrices and thus admits a simple expression as product of the exponentials of each matrix.
As a result, the explicit expressions of $\Ab(z)$ and $\Bb(z)=\Pb \Ab(-z)\Pb$ in terms of the two parameters $a=\sin(\eta)$, $b=\sin(\omega)$ are:
{ \begin{equation}\label{eq:fourthA}\small
\Ab(z)=(z+1)
\left[
 \begin{array}{cc}
b^2z^2+(1-a^2-b^2)z+a^2 & (1-z)\left(b\sqrt{1-b^2} z +a \sqrt{1-a^2}\right)\\
(1-z)\left(a\sqrt{1-a^2} z +b \sqrt{1-b^2}\right) &  a^2 z^2+(1-a^2-b^2)z+b^2
 \end{array}\right]\end{equation}
}
{ \begin{equation}\label{eq:fourthB}\small
\Bb(z)=(1-z)
\left[
 \begin{array}{cc}
-b^2z^2+(1-a^2-b^2)z-a^2 & (1+z)\left(a \sqrt{1-a^2}-b\sqrt{1-b^2} z\right)\\
(1+z)\left(b \sqrt{1-b^2}-a\sqrt{1-a^2}+\right) &  -a^2 z^2+(1-a^2-b^2)z-b^2
 \end{array}\right]\\
 \end{equation}
}

It is interesting to note that the case $a=b=1/2$ produces a symbol $\Ab(z)$ with an extra $(z+1)$ factor, which contains the quartic B-spline symbol (which gives rise to a non orthogonal scalar scheme) on the diagonal. 
However this scheme is essentially diagonal, in the sense explained in \cite{ContiCotroneiSauer07}, and it is equivalent to two independent Daubechies scalar schemes.

\begin{figure}[h]
\begin{center}
\fbox{\includegraphics[width=5.5cm]{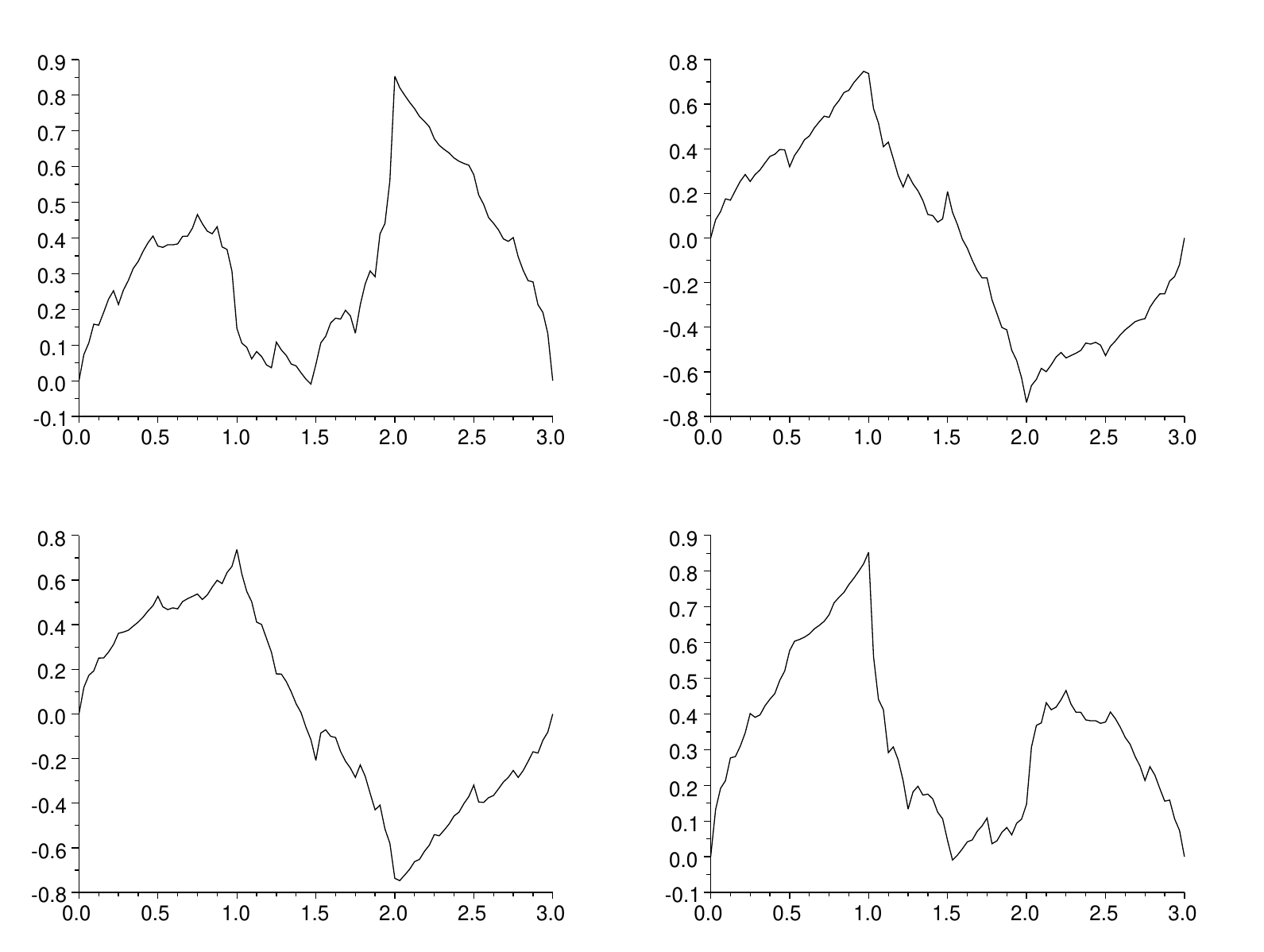}}\ 
\fbox{\includegraphics[width=5.5cm]{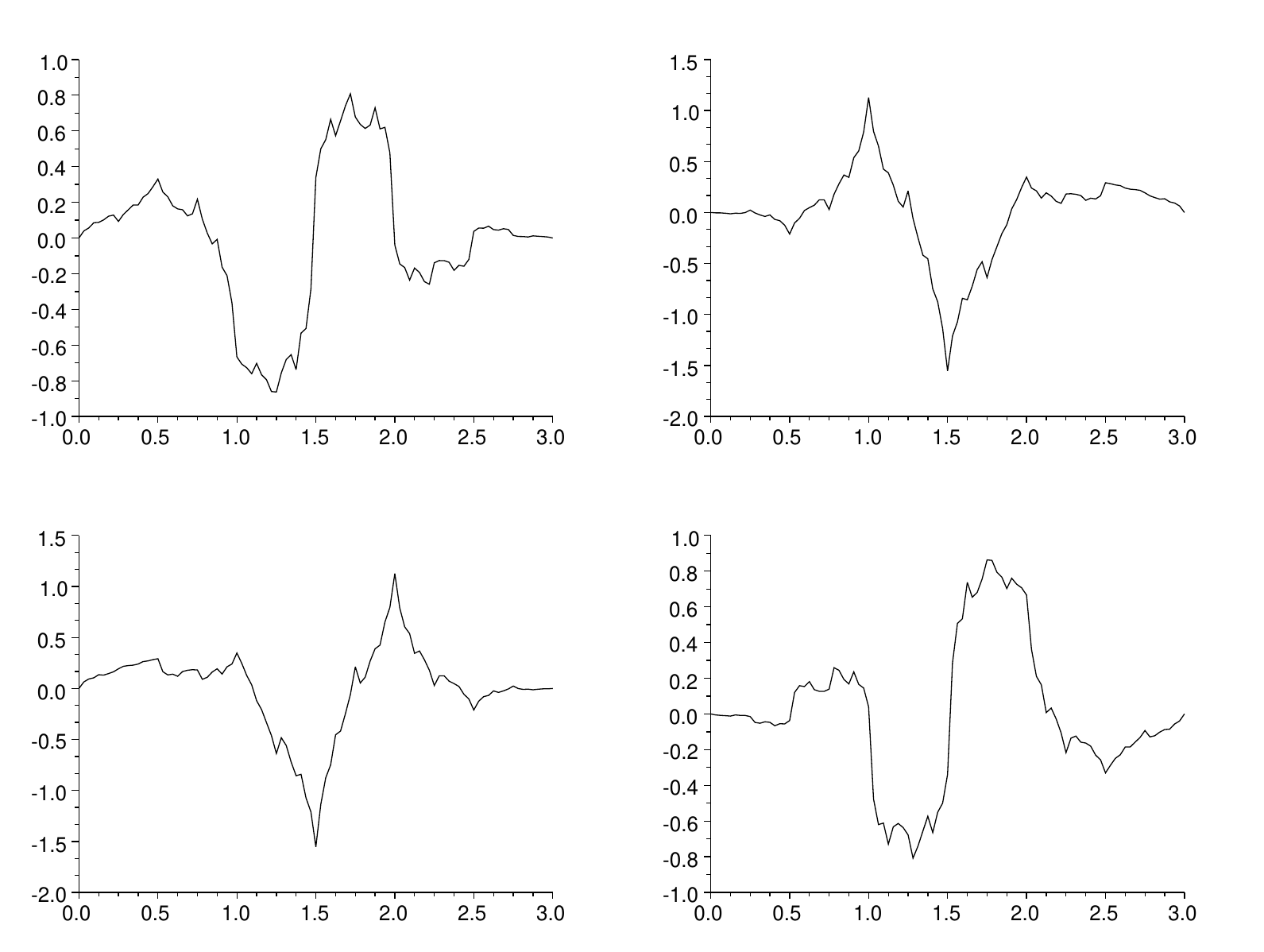}}
\end{center}
\caption{The four components of the 2-channel scaling function (left) and wavelet (right) associated to the two parameter symbols (\ref{eq:fourthA}), (\ref{eq:fourthB}) with $a=0.4$, $b=0.6$.}
\label{fig:fig5}
\end{figure}

\section{Conclusions}
This paper discusses a way to parameterize families of $d\times d$ matrix wavelet filters of full rank type
based on the one-to-one correspondence between QMF systems and orthogonal operators which commute with the shifts by two.
It provides some specific examples, in particular the special construction obtained in terms of elements of the Lie algebra of $SO(2d)$. Explicit expressions for the filters in the case $d=2$ are given, as a result of a local analysis of the parameterization obtained from perturbing the Haar system. This strategy can be used also to generate filters with $d>2$.
Indeed, the idea of the construction is  very general, and can even apply  to the bivariate nonseparable setting.
Furthermore, in all cases,   the parameterization can be exploited to  realize matrix systems  tailored to the specific applications.
Future research work will be carried on in this direction. 
\par
Moreover, we are pursuing a wide experimentation connected to the application of multichannel wavelet filters, including those constructed in this paper, which shows the advantages of this tools versus traditional scalar wavelet techniques. The first results will appear in a forthcoming paper.

 \appendix

 \section{Filter coefficients}
In this appendix, the explicit expressions of the filters taps derived in the Example section are given (see Tables A.1--A.4).

\begin{table}[H]
\renewcommand{\arraystretch}{1.3}
 \caption{Coefficients of the matrix filters associated to the functions plotted in fig.~\ref{fig:examplescalwav}}   \label{tab:CoeffFig2}
    \centering
        \begin{tabular}{cc}\hline
      $k$ & $\Ab(k)$   \\\hline
      $0$& $ \left[ \begin {array}{cc}  0.054311209333498010694&-
 0.16684440354507635408\\  0.047827189604400026824&-
 0.17618585179038651232\end {array} \right]
$      

 \\[0.3cm]
$1$&
$\left[ \begin {array}{cc}  0.087708930185872110064&-
 0.30362455474026860115\\ - 0.088952962252079719437&
 0.29041642693165752616\end {array} \right]
$
\\[0.3cm]
$2$&$\left[ \begin {array}{cc}  0.65527236381392715834&
 0.07789144188237754498\\ - 0.35145174509107688357&
 1.0884769216695398767\end {array} \right]
$
 \\[0.3cm]
$3$& $\left[ \begin {array}{cc}  1.088476921748056695&
 0.35145174524818843912\\ - 0.07789144192861210554&
 0.65527236401729840172\end {array} \right] 
$
  \\[0.3cm]
$4$& $
\left[ \begin {array}{cc}  0.29041642689141252135&
 0.088952962017450863579\\  0.30362455490066007842&
 0.087708930001853805185\end {array} \right]
$
  \\[0.3cm]
 $5$ &$
 \left[ \begin {array}{cc} - 0.17618585197276649584&-
 0.047827189524548346775\\  0.16684440342858505769&
 0.054311209170036901976\end {array} \right] 
$
 \\[0.3cm]\\\hline
       &  $\Bb(k)$ \\\hline
 $0$    
   &   $\left[ \begin {array}{cc}  0.0543112093345566594&-
 0.166844403526512642\\  0.0478271896021416066&-
 0.176185851776288344\end {array} \right]
    $
  \\[0.3cm]
$1$&  $ \left[ \begin {array}{cc}  0.0877089301796203924&-
 0.303624554700031024\\ - 0.088952962258738763&
 0.290416426927014915\end {array} \right] 
        $   
\\[0.3cm]
$2$&$  \left[ \begin {array}{cc} - 0.276396211477648534& 1.1144837374540959
\\ - 0.602889006833760122&- 0.248606647031728134
\end {array} \right]
 $
\\[0.3cm]
$3$
& $ \left[ \begin {array}{cc}  0.248606647088857046&-
 0.602889006560141438\\  1.11448373738822482&
 0.276396211518095014\end {array} \right]
    $
\\[0.3cm]
$4$
& $\left[ \begin {array}{cc} - 0.290416426894247737&-
 0.0889529620232064194\\ - 0.303624554889251552&-
 0.087708930001465138\end {array} \right]
     $
\\[0.3cm]
$5$
& $  \left[ \begin {array}{cc}  0.176185851960222128&
 0.0478271895207618912\\ - 0.166844403430672834&-
 0.0543112091753285287\end {array} \right] 
   $
\\
\hline

        \end{tabular}
\end{table}

\begin{table}[H]
\renewcommand{\arraystretch}{1.3}
 \caption{Coefficients of the matrix filters associated to the symbols (\ref{eq:firstA}), 
 (\ref{eq:firstB}), with $-1\le a \le 1$. The choice $a=\cos(\pi/6)$ leads to the functions plotted in Fig.~\ref{fig:fig3}}  \label{tab:CoeffFig3}
    \centering
        \begin{tabular}{ccc}\hline
      $k$ & $\Ab(k)$  & $\Bb(k)$ \\\hline
      $0$&        $ \left[ \begin {array}{cc} 1-{a}^{2}&a\sqrt {1-{a}^{2}}
\\0&0\end {array} \right] 
$
 &
 $ \left[ \begin {array}{cc} 1-{a}^{2}&-a\sqrt {1-{a}^{2}}
\\0&0\end {array} \right] 
$ \\[0.3cm]
$1$&$ \left[ \begin {array}{cc} 1&0\\ a\sqrt {1-{a}^{2}}&{
a}^{2}\end {array} \right]$
& $\left[ \begin {array}{cc} -1&0\\a\sqrt {1-{a}^{2}}&
-{a}^{2}\end {array} \right]$\\[0.3cm]
$2$&$
\left[ \begin {array}{cc} {a}^{2}&-a\sqrt {1-{a}^{2}}
\\0&1\end {array} \right] $ &
$  \left[ \begin {array}{cc} {a}^{2}&a\sqrt {1-{a}^{2}}
\\0&1\end {array} \right]$
\\[0.3cm]
$3$&$
\left[ \begin {array}{cc} 0&0\\-a\sqrt {1-{a}^{2}}&
1-{a}^{2}\end {array} \right]
$ & $\left[ \begin {array}{cc} 0&0\\-a\sqrt {1-{a}^{2}}&
a^{2}-1\end {array} \right]
$
        \end{tabular}
\end{table}

\begin{table}[H]
\renewcommand{\arraystretch}{1.3}
\caption{Coefficients of the matrix filters associated to the symbols (\ref{eq:secondA}), 
 (\ref{eq:secondB}), with $-1\le a \le 1$. The choice $a=0$ leads to the functions plotted in Fig.~\ref{fig:fig4}} 
    \label{tab:CoeffFig4}
    \centering
        \begin{tabular}{cc}\hline
      $k$ & $4\Ab(k)$  \\\hline
      $0$    
 & $\left[ \begin {array}{cc} 0&0\\ 1-{a}^{2}&-2\,a+
(a+1)\sqrt {2(1-\,{a}^{2})}+3-{a}^{2}\end {array}
 \right]
$       
 
 \\[0.3cm]
$1$&
$\left[ \begin {array}{cc} 3+2\,a-{a}^{2}& \left( a-1 \right)  \left( 
-1-a+\sqrt {2(1-\,{a}^{2})} \right) \\ \left( a-1
 \right) \sqrt {2(1-\,{a}^{2})}&
 2\,{a}^{2}+2+(a+1)\sqrt {2(1-\,{a}^{2})}\end {array} \right]
$
\\[0.3cm]
$2$&$ \left[ \begin {array}{cc} 4&0\\-1+{a}^{2}&{a}^{2}-
(a+1)\sqrt {2(1-\,{a}^{2})}+1+2\,a\end {array} \right]
$
 
\\[0.3cm]
$3$&$\left[ \begin {array}{cc}  \left( a-1 \right) ^{2}&- \left( a-1
 \right)  \left( -1-a+\sqrt {2(1-\,{a}^{2})} \right) 
\\- \left( a-1 \right) \sqrt {2(1-\,{a}^{2})}&-
 \left( a+1 \right)  \left( -2+2\,a+\sqrt {2(1-\,{a}^{2})} \right) 
\end {array} \right] 
$
\\[0.3cm]\\\hline
       &  $4\Bb(k)$ \\\hline
      $0$    
  
 &$\left[ \begin {array}{cc} 0&0\\ -1+{a}^{2}&-2\,a-
(a+1)\sqrt {2(1-\,{a}^{2})}+3-{a}^{2}\end {array}
 \right]
$
 
 \\[0.3cm]
$1$&
$\left[ \begin {array}{cc} -3-2\,a+{a}^{2}&- \left( a-1 \right) 
 \left( 1+a+\sqrt {2(1-\,{a}^{2})} \right) \\ - \left( 
a-1 \right) \sqrt {2(1-\,{a}^{2})}&-2\,{a}^{2}-2+(a+1)\sqrt {2(1-\,{a}^{2})}\end {array} \right] 
$
\\[0.3cm]
$2$&$\left[ \begin {array}{cc} 4&0\\ 1-{a}^{2}&{a}^{2}+
(a+1)\sqrt {2(1-\,{a}^{2})}+1+2\,a\end {array} \right]
$
\\[0.3cm]
$3$
& $\left[ \begin {array}{cc} - \left( a-1 \right) ^{2}& \left( a-1
 \right)  \left( 1+a+\sqrt {2(1-\,{a}^{2})} \right) \\ 
 \left( a-1 \right) \sqrt {2(1-\,{a}^{2})}&- \left( a+1 \right)  \left( 
2-2\,a+\sqrt {2(1-\,{a}^{2})} \right) \end {array} \right] 
$
\\
\hline

        \end{tabular}
\end{table}

\begin{table}[H]
\renewcommand{\arraystretch}{1.3}
 \caption{Coefficients of the matrix filters associated to the symbols (\ref{eq:fourthA}), 
 (\ref{eq:fourthB}), with $-1\le a,b \le 1$. The choice $a=b=1/2$ leads to the functions plotted in Fig.~\ref{fig:fig5}} 
  \label{tab:CoeffFig5}
    \centering
        \begin{tabular}{ccc}\hline
      $k$ & $\Ab(k)$  & $\Bb(k)$ \\\hline
      $0$&   
       $\left[ \begin {array}{cc} {a}^{2}&a\sqrt {1-{a}^{2}}
\\ b\sqrt {1-{b}^{2}}&{b}^{2}\end {array} \right]
$    
 &$\left[ \begin {array}{cc} {a}^{2}&-a\sqrt {1-{a}^{2}}
\\ -b\sqrt {1-{b}^{2}}&{b}^{2}\end {array} \right]
$
  \\[0.3cm]
$1$&
$ \left[ \begin {array}{cc} 1-{b}^{2}&b\sqrt {1-{b}^{2}}
\\ a\sqrt {1-{a}^{2}}&1-{a}^{2}\end {array} \right] 
$
& $\left[ \begin {array}{cc} -1+{b}^{2}&b\sqrt {1-{b}^{2}}
\\ a\sqrt {1-{a}^{2}}&-1+{a}^{2}\end {array} \right]$
\\[0.3cm]
$2$&$\left[ \begin {array}{cc} 1-{a}^{2}&-a\sqrt {1-{a}^{2}}
\\ -b\sqrt {1-{b}^{2}}&1-{b}^{2}\end {array} \right]
$
 &$ \left[ \begin {array}{cc} 1-{a}^{2}&a\sqrt {1-{a}^{2}}
\\ b\sqrt {1-{b}^{2}}&1-{b}^{2}\end {array} \right]
$

\\[0.3cm]
$3$&$ \left[ \begin {array}{cc} {b}^{2}&-b\sqrt {1-{b}^{2}}
\\ -a\sqrt {1-{a}^{2}}&{a}^{2}\end {array} \right] 
$&
$\left[ \begin {array}{cc} -{b}^{2}&-b\sqrt {1-{b}^{2}}
\\ -a\sqrt {1-{a}^{2}}&-{a}^{2}\end {array} \right] 
$

        \end{tabular}
\end{table}


%% References


\bibliographystyle{elsarticle-num}


\end{document}